\documentclass[12pt]{amsart}
\usepackage{mathrsfs}
\usepackage[colorlinks]{hyperref}
\usepackage{xcolor}

\usepackage{setspace}

\usepackage{mathscinet}
\usepackage{latexsym}
\usepackage{amsthm}
\usepackage{amssymb}
\usepackage{amsfonts}
\usepackage{amsmath}
\usepackage{verbatim}

\newtheorem{theorem}{Theorem}[section]
\newtheorem{lemma}[theorem]{Lemma}
\newtheorem{proposition}[theorem]{Proposition}
\newtheorem{corollary}[theorem]{Corollary}

\theoremstyle{definition}
\newtheorem{definition}[theorem]{Definition}

\theoremstyle{remark}
\newtheorem{remark}[theorem]{Remark}
\numberwithin{equation}{section}
\def\d{{\rm d}}
\newcommand{\e}{{\rm e}}

\begin{document}
\title{Heath--Jarrow--Merton  model with linear volatility}

\keywords{HJM model, forward rate curve, bond market}
\author[Szymon Peszat]{Szymon Peszat*}
\thanks{* Corresponding author}
\address{Szymon Peszat:  Institute of Mathematics, Jagiellonian University, {\L}ojasiewicza 6, 30--348 Krak\'ow, Poland}
\email{napeszat@cyf-kr.edu.pl}

\author{Jerzy Zabczyk}
\address{Jerzy Zabczyk: Institute of Mathematics, Polish Academy of Sciences, ul. \'Sniadeckich~8, 00-956 Warsaw, Poland}
\email{ jerzy@zabczyk.com}

\thanks{}

\subjclass[2020]{60H30, 93E20, 91G30, 91G10}

\maketitle

\begin{abstract}
We consider the  Heath--Jarrow--Morton model of forward rates processes with linear volatility. The noise is either a Wiener  or a pure jump L\'evy process. We provide formulae for the forward rate processes, and discus the problem of their  global in time existence. 
\end{abstract}
\section{Introduction}
In the classical  \emph{Heath--Jarrow--Morton model},  the forward rates are  It\^o's processes (see \eqref{E21}). Following \cite{Barski-Zabczyk-1,Barski-Zabczyk,   Bjork-Christensen, Bjork-Kabanov-Runggaldier,  Bjork-DiMassi-Kabanov-Runggaldier, Eberlein-Raible,J.Z, Jakubowski-Zabczyk}) we consider the case where the driving noise is a L\'evy process $L$ taking values in $\mathbb{R}^d$. It is known (see also Section \ref{S2}) that under the Musiela parametrization the forward rate is the solution to the stochastic partial differential equation 
\begin{equation}\label{E11}
\begin{aligned}
\d r (t,x)&= \left( \frac {\partial r}{\partial x}(t,x)+ b(t,x,r(t,x))\right)\d t \\
&\qquad +\langle \sigma(t,x,r(t-,x)) , \d L(t)\rangle_{\mathbb{R}^d}. 
\end{aligned}
\end{equation}

In the paper we are concerned with the case where the \emph{volatility} $\sigma$ depends linearly on the forward rate $r$; that is 
\begin{equation}\label{E12}
\sigma(t,x,r(t,x))= g(t)r(t,x), 
\end{equation}
where $g$ is an $\mathbb{R}^d$-valued  predictable process. In this case we provide an analytic form of the solution (see Theorems \ref{T34},  \ref{T41}, and \ref{T71}). 

In the  paper we extend results of \cite{Peszat-Zabczyk-1} and Chapter 20 of \cite{Peszat-Zabczyk}, but what is more important we fill some gaps in arguments from \cite{Peszat-Zabczyk-1, Peszat-Zabczyk} providing  correct equation  for the primitive function, see Proposition  \ref{P32}. 

\section{Heath--Jarrow--Morton model}\label{S2}
Let us denote by $B(t,S)$ the price at time $t$ of a bond paying $1$ at time $S$. Assume that the  \emph{forward rates} 
$$
f(t,S)= -\frac{\partial }{\partial S} \log B(t,S), \qquad 0\le t\le S, 
$$
are given by the stochastic equation 
\begin{equation}\label{E21}
\d f(t,S)= F(t,S,f(t,S))\d t + \langle G(t,S,f(t-,S)),  \d L(t)\rangle_{\mathbb{R}^d}, 
\end{equation}
driven by an $\mathbb{R}^d$-valued  L\'evy  process $L$ defined on a filtered probability space $(\Omega,\mathfrak{F},(\mathfrak{F}_t),\mathbb{P})$. In \eqref{E21}, $F$ and $G$ are $\mathbb{R}$  and $\mathbb{R}^{d}$ -valued processes. Clearly 
$$
B(t,S)=\e^{-\int_t^S f(t,y)\d y}.
$$ 
Let us rewrite the prices and forward rates in the so-called \emph{Musiela parametrization}
$$
P(t, x):= B(t,t+x), \quad r(t, x):=f(t,t+x), \qquad x\ge 0.  
$$
Then 
$$
P(t,x)= \e^{-\int_0^x r(t,y)\d y}, \qquad r(t,x)=-\frac{\partial }{\partial x} \log P(t,x), 
$$
and 
\begin{align*}
r(t,\cdot)&= \mathcal{S}(t)r(0,\cdot )+ \int_0^t \mathcal{S}(t-s)b(s,\cdot, r(s,\cdot) )\d s \\
&\qquad + \int_0^t \mathcal{S}(t-s)\langle \sigma(s-,\cdot, r(s-, \cdot)),\d L(s)\rangle_{\mathbb{R}^d}, 
\end{align*}
where 
\begin{align*}
b(t,x, r(t,x))&:= F(t,t+x, f(t,t+x))= F(t,t+x, r(t,x)), \\ 
\sigma(t,x, r(t,x))&:= G(t,t+x,f(t,t+x))= G(t,t+x,r(t,x)),
\end{align*}
and for  a real-valued function $\psi$, 
$$
\mathcal{S}(t)\psi(x)=\psi(x+t),\qquad  t\ge 0,\ x\in \mathbb{R}. 
$$
Since $\mathcal{S}$ is the semigroup generated by the derivative operator $\frac {\partial }{\partial x}$ the forward rate  $r$ is the mild solution to stochastic partial differential equation \eqref{E11}. In that follows we assume \eqref{E12}; that is that the volatility $\sigma$ depends linearly on the forward rate $r$. We are looking for a formula for the solution. Therefore without any loss of generality we may assume that $\mathbb{P}$ is the martingale measure (see e.g. \cite{Barski-Zabczyk,Peszat-Zabczyk}). This implies that $b$ as a function of $r$ is determined by $g$ and the noise $L$,  by the so-called \emph{Heath--Jarrow--Morton condition}.  Namely, if $L$ is a standard Wiener process then,  see \cite{H-J-M}, 
\begin{equation}\label{E22}
b(t,x)= \|g(t)\|_{\mathbb{R}^d}^2\,  r(t,x)\int_{0}^{x}r(t,y)\d y. 
\end{equation}
In the more general  case of $L$ being a L\'evy process, $b$ is determined by  $\sigma$ through  the Laplace transform of $L$, namely,  see e.g. \cite{Barski-Zabczyk-1,Barski-Zabczyk,   Bjork-Christensen, Bjork-Kabanov-Runggaldier,  Bjork-DiMassi-Kabanov-Runggaldier, Eberlein-Raible,J.Z, Jakubowski-Zabczyk}, 
\begin{equation}\label{E23}
b(t,x)=\frac{\partial }{\partial x} J\left(g(t)\int_0^xr(t,y)\d y\right),
\end{equation}
where 
\begin{equation}\label{E24}
J(\xi):= \log \mathbb{E}\exp\left\{-\langle \xi, L(1)\rangle_{\mathbb{R}^d}\right\}, \qquad \xi\in \mathbb{R}^d. 
\end{equation}
In the following particular cases:  

\smallskip 
\noindent\textbf{Gaussian case} Assume that $L=W$, where $W$ is a Wiener process with covariance $Q$; 
$$
\mathbb{E}\, \langle W(t),\xi\rangle_{\mathbb{R}^d}  \langle W(s),\eta\rangle_{\mathbb{R}^d}= \min\{s,t\}\langle Q\xi,\eta\rangle_{\mathbb{R}^d}. 
$$
we have 
$$
J(\xi)= \frac{1}{2}\langle Q\xi,\xi\rangle_{\mathbb{R}^d}. 
$$

Note that if $W$ is standard, then $Q$ is the identity operator, and 
\begin{equation}\label{E25}
\frac{\partial }{\partial x} J\left(g(t)\int_0^xr(t,u)\d u\right)=\|g(t)\|_{\mathbb{R}^d}^2\,  r(t,x)\int_{0}^{x}r(t,u)\d u, 
\end{equation}
as required by \eqref{E22}. 

\smallskip
\noindent\textbf{$\alpha$-subordinator case}
Assume that $d=1$.  Let    $L$ be  the $\alpha$-stable subordinator, where $\alpha \in (0,1)$, that is 
\begin{equation}\label{E26}
L(t)=\int_0^t \int_{\mathbb{R}^d}\eta \pi(\d s, \d \eta), 
\end{equation}
$\pi$  is a Poisson measure on $[0,+\infty)\times \mathbb{R}^d$ with intensity $\d s\nu(\d \eta)$, 
$$
\nu(\d \eta)= c_\alpha \vert \eta \vert ^{-1-\alpha}\chi_{\{\eta >0\}}. 
$$
Then, see e.g. \cite{Applebaum}[pp. 51-52],    
$$
J(\xi)=\begin{cases}- \xi  ^\alpha,&\text{$\xi\ge 0$},\\
+\infty,&\text{$\xi< 0$}. 
\end{cases}
$$ 

\smallskip
\noindent\textbf{Jump martingale case} Assume that  
$$
L(t)=\int_0^t \int_{\mathbb{R}^d}\eta \left(\pi(\d s, \d \eta)- \d s \nu(\d \eta)\right), 
$$
$\pi$  is a Poisson measure on $[0,+\infty)\times \mathbb{R}^d$ with intensity $\d s\nu(\d \eta)$. Then 
\begin{equation}\label{E27}
J(\xi)= \int_{\mathbb{R}^d}\left[ \e^{-\langle \xi,\eta\rangle_{\mathbb{R}^d}}-1 +\langle \xi,\eta\rangle_{\mathbb{R}^d}\right]\nu (\d \eta). 
\end{equation}
In order to assure positivity of $r(t,x)$ we assume that $\nu$ is concentrated on $[-m,+\infty)^d$ for some $m\ge 0$,  and that $g(t)\in [0,m^{-1})^d$ for $t\ge 0$, $\mathbb{P}$-a.s., see Theorem 20.14 from \cite{Peszat-Zabczyk}.  
\section{General result}
Assume that in \eqref{E11}, $L$ is a  L\'evy process with L\'evy exponent $J$ given by \eqref{E24}. Then by \eqref{E12} and \eqref{E22} the forward rate processes $r(t, x)$ are defined by the following first order stochastic partial differential equation (Heath--Jarrow--Morton--Musiela equation or HJMM equation for short)
\begin{equation}\label{E31}
\left\{ 
\begin{aligned}
\d r(t,x)&= \left[ \frac{\partial r}{\partial x}(t,x)+ \frac{\partial }{\partial x} J\left(g(t)\int_0^xr(t,y)\d y\right)\right] \d t \\
&\quad + r(t-,x) \langle g(t), \d L(t)\rangle_{\mathbb{R}^d}, \\
r(0, x)&=r_0(x). 
\end{aligned}
\right. 
\end{equation} 
We assume that $g$ is a measurable adapted process satisfying 
$$
\mathbb{P}\left( \int_0^T\|g(s)\|^2_{\mathbb{R}^d}\d s <+\infty\right), \qquad \forall\, T<+\infty. 
$$
Equation \eqref{E31} can be considered on different state spaces $E$. For example  on  $E=L^1_{\text{loc}}([0,+\infty))$ or $E=L^1([0,+\infty))$.   On these spaces  the stochastic integral is well-defined (see e.g. \cite{Brzezniak-Veraar,veraar1}), $\frac{\partial }{\partial x}$ defines $C_0$-semigroup $\mathcal{S}(t)$, $t\ge 0$, of translations $\mathcal{S}(t)\psi(x)=\psi(x+t)$, and the nonlinear map 
$$
E \ni \psi \mapsto \psi(\cdot )\int_0^{\cdot}\psi(y)\d y\in E
$$
is locally Lipschitz. 

In the present paper we deal with the \emph{strong solution}.  Let $\eth\not \in \mathbb{R}$. 
\begin{definition} A random field $r$ on $[0,+\infty)\times [0,+\infty)$ taking values in $\mathbb{R}\cup\{\eth\}$  is called a \emph{strong solution} to \eqref{E31}, if for any $x\in [0,+\infty)$ there is a stopping time $\tau(x,r_0)$ such that: 
\begin{itemize}
\item $r\in C^{0,1}(\mathcal{D})$, $\mathbb{P}$-a.s., where 
$$
\mathcal{D}=\{(t,x)\in [0,+\infty)\times[0,+\infty)\colon t<\tau(x,r_0)\},
$$
\item $r(t,x)=\eth $ for $(t,x)\not \in \mathcal{D}$,
\item for any $x\in [0,+\infty)$, $r(\cdot,x)$ is measurable adapted, and 
\begin{align*}
r(t,x)&= r_0(x)+ \int_{0}^t \left[ \frac{\partial r}{\partial x}(s,x)+  \frac{\partial }{\partial x} J\left(g(s)\int_0^xr(s,y)\d y\right)\right]\d s\\
&\qquad + \int_0^t r(s-,x)\langle g(s),\d L(s)\rangle_{\mathbb{R}^d}\quad \text{for} \ t<\tau(x,r_0), \ \mathbb{P}-a.s.
\end{align*}
\end{itemize}
\end{definition}

Denote by $u(t,\cdot)$ the following  primitive of $r(t,\cdot)$; 
$$
u(t,x):= \int_0^x r(t, y)\d y. 
$$

\begin{proposition}\label{P32}
Assume that for any $x\ge 0$, $\tau(x,r_0)\le \tau(0,r_0)$. Then the primitive $u$ satisfies the following equation
\begin{equation}\label{E32}
\begin{aligned}
\d u(t, x)&= \left[ \frac{\partial u}{\partial x} (t,x)- \frac{\partial u}{\partial x}(t,0)+ J\left(g(t)u(t,x)\right)\right]\d t \\
&\qquad +  u(t-,x)\langle g(t),\d L(t)\rangle_{\mathbb{R}^d},\\
u(t,0 )&=0,\\
 u(0,x)&=\int_0^x r_0(y)\d y.  
\end{aligned}
\end{equation}
\end{proposition}
\begin{proof} By definition of the solution for any $x\ge 0$ there is a stopping time $\tau(x,r_0)$ such that for  $t<\tau(x,r_0)$, 
\begin{align*}
&\frac{\partial }{\partial x}\left\{ u(t,x)-u(0,x)- \int_0^t \left[  \frac{\partial u}{\partial x}(s,x)+ J\left(g(s)u(s,x)\right)\right]\d s\right. \\
&\qquad \left.  - \int_0^t u(s-,x)\langle g(s),\d L(s)\rangle_{\mathbb{R}^d}\right\}=0.
\end{align*}
Therefore for arbitrary $0\le t$ there is an independent  on $x\colon t<\tau(x,r_0)$ constant $C(t)$ such that 
\begin{align*}
u(t,x)&= u(0,x)+ \int_0^t \left[ \frac{\partial u}{\partial x}(s,x) + J\left(g(s)u(s,x)\right)\right]\d s \\
&\quad + \int_0^t u(s-,x)\langle g(s),\d L(s)\rangle_{\mathbb{R}^d}+ C(t). 
\end{align*}
Let $t<\tau(x,r_0)\le \tau(0,r_0)$. Then 
\begin{align*}
u(t,0)&= u(0,0)+ \int_0^t \left[ \frac{\partial u}{\partial x}(s,0) + J\left(g(s)u(s,0)\right)\right]\d s \\
&\quad + \int_0^t u(s-,0)\langle g(s),\d L(s)\rangle_{\mathbb{R}^d}+ C(t). 
\end{align*}
Since by definition $u(s,0)=0$ and $J(0)=0$,  we have 
$$
C(t)=- \int_0^t \frac{\partial u}{\partial x}(s, 0)\d s, 
$$
which gives the desired equation on $u$. 
\end{proof}

In order to solve the equation for $u$ let us consider the flow $Y$ to the  ordinary stochastic differential equation 
$$
\d y(t)= J\left(g(t)y(t)\right)\d t + y(t-)\langle g(t),\d L(t)\rangle_{\mathbb{R}^d},
$$
that is we are looking for random field $Y$ such that 
\begin{equation}\label{E33}
\begin{aligned}
\d Y(t,x) &=  J\left(g(t)Y(t,x)\right)\d t + Y(t-,x)\langle g(t), \d L (t)\rangle_{\mathbb{R}^d},\\
Y(0,x) &= x.
\end{aligned}
\end{equation}
The solution can be local in time, that is defined for $t<\tau(x)$. 
\begin{proposition}\label{P33}
Let $Y=Y(t,x),\quad t<\tau(x)$, be the solution to \eqref{E33}. Assume that the blow up time $\tau(x)$ depends continuously on $x$, and that for $t<\tau(x)$, $Y(t,x)$ is continuously differentiable in $x$,  then
$$
u(t, x) =Y\left(t, \int_t^{t+x} r_0(y)\d y\right) ,\quad t\ge 0,\ x\ge 0,
$$
is a solution  \eqref{E32}. The solution $u$ is defined up to the time $t< \tau\left(\int_t^{t+x}r_0(y)\d y \right)$. 
\end{proposition}
\begin{proof}
Note that $u(0,0)=0$. Therefore 
$$
Y\left(0,\int_0^x r_0(y)\d y \right)= \int_0^x r_0(y)\d y  = u(0,x),
$$
as required.  We claim that
$$
\begin{aligned}
&\d Y\left(t, \int_t^{t+x} r_0(y)\d y\right) = \d Y(t, u(0,t+x)-u(0,t))\\
&= \left(\frac{\partial }{\partial x} Y(t, u(0,t+x)-u(0,t))- \frac{\partial }{\partial x} Y(t, u(0,t)-u(0,t))\right)\d t\\
&\quad +  J\left(g(t)Y(t,u(0,t+x)-u(0,t))\right)\d t
\\
&\quad + Y(t-, u(0,t+x)-u(0,t))\langle g(t),\d L(t)\rangle_{\mathbb{R}^d}.
\end{aligned}
$$
Let $\psi\colon [0,+\infty)\times [0,+\infty)\mapsto \mathbb{R}$ be a differentiable function of both variables. Given $x>0$ consider the process $\left(Y(t, \psi(t,x)),\,t\ge 0\right)$.  Then for any partition $0=t_0<t_1<\ldots<t_N=t$,
$$
Y(t, \psi(t,x))- Y(0, \psi(0,x))= I_1+I_2,
$$
where
\begin{align*}
I_1 &:= \sum_n \left( Y(t_{n+1}, \psi(t_{n+1},x)) - Y(t_{n+1}, \psi(t_n,x)) \right),\\
I_2 &:= \sum_n \left(Y(t_{n+1}, \psi(t_n,x)) - Y( t_n, \psi(t_n,x)) \right).
\end{align*}
We have 
\begin{align*}
I_1 &=\sum_n \left[\frac{\partial Y}{\partial x}\left(t_{n+1}, \psi(t_{n+1},x) + \varepsilon _n\left(\psi(t_n,x)-\psi(t_{n+1},x)\right)\right)\right.\\
&\qquad  \left.\phantom{\frac{\partial Y}{\partial x}} \times  \left(\psi(t_{n+1},x)-\psi(t_n,x)\right)\right]\\
&=\sum_n \left[ \frac{\partial Y}{\partial x}\left(t_{n+1}, \psi(t_{n+1},x)+\varepsilon _n\left(\psi(t_n,x)-\psi(t_{n+1},x)\right)\right)\right.\\
&\qquad \left. \phantom{\frac{\partial Y}{\partial x}}\times  \frac{\partial \psi}{\partial t}\left(t_n+ \eta_n(t_{n+1}-t_n) ,x\right)(t_{n+1}-t_n)\right],
\end{align*}
where $|\varepsilon_n|\le 1$ and $|\eta_n|\le 1$. Taking into account the continuous dependence of $\frac{\partial Y}{\partial x}(t,x)$ on the second variable we obtain
$$
I_1\to \int_0^t \frac{\partial Y}{\partial x}(s, \psi(s,x))\frac{\partial \psi}{\partial t}(s,x)\d s\quad \text{as $n\uparrow \infty$}.
$$
Let 
$$
Z(t):= \int_0^t \langle g(s), \d L(s)\rangle. 
$$
Since $Y$ satisfies \eqref{E33},
$$
I_2 =\sum_n \left( \int_{t_n}^{t_{n+1}} J\left( g(s)Y(s-, \psi(t_{n+1},x))\right) \d s +\int_{t_n}^{t_{n+1}} Y(s-, \psi(t_n,\xi)) \d Z(s)\right)
$$
and therefore
$$
I_2 \to  \int_0^ t J\left( g(s)Y(s-,\psi(s,x))\right)\d s + \int_0^t Y(s-,\psi(s,x))\d Z(s).
$$
\end{proof}
As a consequence of Propositions \ref{P32}, \ref{P33} we have the following result. 
\begin{theorem}\label{T34}
Assume that \eqref{E33} has a local in time solution $Y=Y(t,x),\quad t<\tau(x)$. Assume that $\tau(x)\le \tau(0)$, $\tau(x)$ depends continuously on $x$, and that for $t<\tau(x)$, $Y(t,x)$ is continuously differentiable in $x$,  then for any $r_0$, 
\begin{align*}
r(t,x) &=\frac{\partial u}{\partial x}(t, x)=\frac{\d}{\d x}  Y\left(t, \int_t^{t+x} r_0(y)\d y\right)
\end{align*}
is a strong solution to \eqref{E31}. 
\end{theorem}

\section{Gaussian case}
Assume that in \eqref{E11}, the L\'evy process is a standard Wiener process (denoted by $W$). Then by \eqref{E25} and \eqref{E31} the forward rate processes $r(t, x)$ are defined by the following first order stochastic partial differential equation 
\begin{equation}\label{E41}
\left\{ 
\begin{aligned}
\d r(t,x)&= \left( \frac{\partial r}{\partial x}(t,x)+ \|g(t)\|^2_{\mathbb{R}^d}r(t,x)\int_0^x r(t,y)\d y\right)\d t \\
&\quad + r(t,x) \langle g(t), \d W(t)\rangle_{\mathbb{R}^d}, \\
r(0, x)&=r_0(x). 
\end{aligned}
\right. 
\end{equation} 

Set 
\begin{equation}\label{E42}
M_g(t):= \exp\left\{ -\int_0^t \langle g(s), \d W(s)\rangle_{\mathbb{R}^d}+ \frac{1}{2}\int_0^t \|g(s)\|^2_{\mathbb{R}^d}\d s\right\}. 
\end{equation}
\begin{theorem}\label{T41}
Assume that $r_0\in C^1([0,+\infty))$ is a strictly positive function. Then there  is a unique strong solution $r$ to \eqref{E41}, $r$ is strictly positive  and it is given by the formula
\begin{align*}
&r(t,x)\\
&=\frac{\d }{\d x} \left[ \left( \int_t^{t+x} r_0(y) \d y\right)^{-1} - \frac{1}{2}\int_0^t\ M^{-1}_g(s)\|g(s)\|^2_{\mathbb{R}^d}\d s\right]^{-1}M^{-1}_g(t)\\
%&= M^{-1}_g(t) \left[ \left(\int_t^{t+x} r_0(y)\d y\right)^{-1} - \frac{1}{2}\int_0^t M^{-1}_g(s)\|g(s)\|^2_{\mathbb{R}^d}\d s\right]^{-2}\\
%&\qquad \times  \left( \int_t^{t+x} r_0(y)\d y\right)^{-2}r_0(t+x).
\end{align*}
The solution is defined on the set 
$$
\mathcal{D}:= \left\{ (t, x)\in [0,+\infty)\times [0,+\infty)\colon 0\le t<\tau(x,r_0)\right\},
$$
where 
$$
\tau(x,r_0)=\inf\left\{t\ge 0\colon \int_t^{t+x} r_0(y)\d y =   2\left(\int_0^t M^{-1}_g(s)\|g(s)\|^2_{\mathbb{R}^d}\d s\right)^{-1}\right\}. 
$$
\end{theorem}

Note that $\tau(0,r_0)=+\infty$. The uniqueness follows as the nonlinear term is locally Lipschitz continuous. Thus the theorem is a direct consequence Theorem \ref{T34} and Lemma \ref{L42} below. Namely, consider  the following Bernoulli type equation
$$
\d y(t)= \frac{\|g(t)\|^2_{\mathbb{R}^d}}{2}y^2(t)\d t +  y(t)\langle g(t), \d W(t)\rangle_{\mathbb{R}^d},
$$ 
and let $Y$ be the corresponding flow, that is  
\begin{equation}\label{E43}
\begin{aligned}
\d Y(t,x)&= \frac{\|g(t)\|^2_{\mathbb{R}^d}}{2}Y^2(t,x)\d t +  Y(t,x)\langle g(t), \d W(t)\rangle_{\mathbb{R}^d},\\
Y(0, x)&=x.
\end{aligned} 
\end{equation}
\begin{lemma} \label{L42}
The unique  solution to \eqref{E43}  is given by the formula
\begin{equation}\label{E44}
Y(t,x)= \left( \frac{1}{x} - \frac{1}{2}\int_0^t M^{-1}_g(s)\|g(s)\|^2_{\mathbb{R}^d}\d s\right)^{-1}M^{-1}_g(t),
\end{equation}
where $M_g$ is given by \eqref{E42}. As in deterministic cases $Y$  is only local in time and defined up to 
$$
\tau(x):= \inf\left \{t\colon  \int_0^t M^{-1}_g(s)\|g(s)\|^2_{\mathbb{R}^d}\d s =\frac{2}{x}\right\}.
$$
\end{lemma}
\begin{proof} Since the coefficients are locally Lipschitz the uniqueness follows.  To find a formula for the solution we will use the same transformation as for the deterministic Bernoulli equation. Namely, let 
$$
V(t,x)= \frac{1}{Y(t, x)}=Y^{-1}(t,x).  
$$
With a slight abuse of the It\^o formula we have 
\begin{align*}
\d V(t,x)&= \left[ -Y^{-2}(t,x)\frac{\|g(t)\|^2_{\mathbb{R}^d}}{2}Y^2(t,x)+ Y^{-3}(t,x)\|g(t)\|^2_{\mathbb{R}^d}Y^2(t,x)\right]\d t \\
&\qquad - Y^{-2}(t,x)Y(t,x)\langle g(t), \d W(t)\rangle_{\mathbb{R}^d}\\
&= \|g(t)\|^2_{\mathbb{R}^d}\left( -\frac{1}{2}+ V(t,x)\right)\d t -V(t,x)\langle g(t), \d W(t)\rangle_{\mathbb{R}^d}. 
\end{align*}
We are going to solve the equation on $V$ using the variation of constant formula. Note that $M_g$ is  the solution to the homogeneous equation; that is  
$$
\d M_g(t)= \|g(t)\|^2_{\mathbb{R}^d}M_g(t)\d t -M_g(t)\langle g(t), \d W(t)\rangle_{\mathbb{R}^d},\qquad M_g(0)=1. 
$$
We are looking for $V$ of the form 
$$
V(t,x)= c(t,x)M_g(t),
$$
where $c(\cdot,x)$ is an absolutely continuous function. Then 
\begin{align*}
\d V(t,x)&= c(t,x)\d M_g(t)+ M_g(t)\frac{\partial c}{\partial t}(t,x)\d t.
\end{align*}
Therefore 
$$
\frac{\partial c}{\partial t}(t,x)= -\frac{\|g(t)\|^2_{\mathbb{R}^d}}{2M_g(t)}. 
$$
Taking into account the initial value condition 
$$
V(0,x)= \frac{1}{Y(0,x)}= \frac{1}{x},
$$
we obtain 
$$
V(t,x)= \left( \frac{1}{x} -\frac{1}{2} \int_0^t M^{-1}_g(s)\|g(s)\|^2_{\mathbb{R}^d}\d s\right)M_g(t). 
$$
 \end{proof}
 
 Let $Y$ be the solution to equation \eqref{E43}. Then by Proposition \ref{P33}, the primitive $u(t,x)=\int_0^xr(t,y)\d t$ is given by formula 
 $$
 u(t,x)= Y\left(t,\int_t^{t+x} r_0(y)\d y\right). 
 $$
 Therefore after differentiation we arrive at the formula on $r$ from Theorem \ref{T41}. 
 
\section{Original formulation and existence of regular solution} 
If $r$ satisfies \eqref{E41}, then the forward rates $f(t,S)$, $S\ge t$,  solve 
\begin{equation}\label{E51}
\begin{aligned}
f(t,S)&=f(0,S)+ \int_0^t \|g(s)\|_{\mathbb{R}^d}^2 f(s,S) \int_s^S f(s, u)\d u \d s\\
&\qquad + \int_0^t f(s,S)\langle g(s), \d W(s)\rangle_{\mathbb{R}^d}. 
\end{aligned}
\end{equation}
By Theorem \ref{T41}, we have 
$$
 \begin{aligned}
 f(t,S)&= f(0,S)\left[ \left( \int_t^Sf(0,s)\d s\right)^{-1} -\frac{1}{2}\int_0^t M_g^{-1}(s)\|g(s)\|_{\mathbb{R}^d}^2 \d s\right]^{-2}\\
 &\qquad \times \left( \int_t^Sf(0,s)\d s\right)^{-2} M_g^{-1}(s). 
 \end{aligned}
$$
 
Let us fix a finite time horizon $T^*<+\infty$. Assume for simplicity that $d=1$ and $g\equiv 1$.  We are looking for a continuous random field 
$$
f\colon \{0\le t\le S\le T^*\}\mapsto [0,+\infty), 
$$
such that for each $S$, $f(t,S)$, $0\le t\le S$ solves \eqref{E51}. Such field is called a \emph{regular solution}, see \cite{Barski-Zabczyk}. The following result  was established by Morton \cite{Morton}, see also \cite{Barski-Zabczyk}, Chapter 14 for a  general L\'evy case. 
\begin{theorem}
For any $\varepsilon \in (0,1)$ there is a $K>0$ such that if $f(0,S)\ge K$ for $S\in [0,T^*]$, then with probability bigger that $1-\varepsilon$ the regular solution does not exist. 
\end{theorem}

We will show that the theorem above follows from our formula of solution. In fact the following stronger result holds true. 
\begin{theorem}
For arbitrary  $\varepsilon \in (0,1)$, and $t,S$ such that $0<t< S\le T^*$ there is a $K>0$ such that if $f(0,T)\ge K$ for $T\in [t,S]$, then with probability bigger that $1-\varepsilon$ the regular solution $f(\cdot ,S)$ does not exist on the time  interval $[0,t]$. 
\end{theorem}
\begin{proof} Let us fix $0<t<S\le T^*$. Set 
$$
 \eta(s,S):= \left( \int_s^Sf(0,v)\d v\right)^{-1} -\frac{1}{2}\int_0^s M_1^{-1}(v)\d v.
 $$
It is enough  to show that 
$$
\mathbb{P}\left\{ \exists\, 0\le s\le t\colon \eta (s,S)=0\right\}\ge 1-\varepsilon. 
$$
 Assume that $f(0,T)\ge K>0$ for all $T\in [t,S]$. Then, 
 $$
 \left(\int_s^S f(0,v)\d v\right)^{-1}\le \left[ (S-t)K\right]^{-1} \qquad \text{for $s\le t$,}
 $$
 and it is enough to show that for $K$ large enough 
 $$
 \mathbb{P}\left\{  \frac{1}{2}\int_0^t M_1^{-1}(v)\d v\ge \left[ (S-t)K\right]^{-1} \right\}\ge 1-\varepsilon. 
$$
To justify the above note that 
\begin{align*}
\int_0^t M_1^{-1}(v)\d v&= \int_0^t \e^{W(v)- \frac{v}{2}}\d v\ge \e^{-\frac{t}{2}} \int_0^t \e^{W(v)}\d v\ge t \e^{-\frac{t}{2}} \e^{\frac{1}{t}\int_0^t W(v)\d v}. 
\end{align*}
Therefore 
\begin{align*}
&\mathbb{P}\left\{ \int_0^t M_1^{-1}(v)\d v\ge  \left[ (S-t)K\right]^{-1} \right\}\\
&\qquad \ge \mathbb{P}\left\{ \int_0^t W(v)\d v\ge - t \e^{\frac{t}{2}}\log \left[  t  (S-t)K \right] \right\}. 
\end{align*}
\end{proof}
  \section{Jump martingale case}
Let $L$ be a  jump L\'evy martingale  with the L\'evy measure $\nu$. Recall that $J$ is given by \eqref{E27}.  We assume that there is an $m\ge 0$ such that 
\begin{equation}\label{E61}
\begin{aligned}
&\text{$\nu$ has support in $[-m,+\infty)^d$, $g(t)\in [0,m)^d$,  and }\\
&\int_{\mathbb{R}^d} \max\left\{\|\eta\|_{\mathbb{R}^d},  \|\eta\|^2_{\mathbb{R}^d}\right\}\nu(\d \eta)<+\infty.
\end{aligned}
\end{equation}
Clearly  \eqref{E61} ensures that $J\colon (0,+\infty)^d\mapsto \mathbb{R}$ is of the class $C^2$ with bounded derivatives of the first order. Moreover, see Theorem 20.14 from \cite{Peszat-Zabczyk},  \eqref{E61} preserves positivity; that is if $r_0(x)\ge 0$ for $x\ge 0$, then $r(t,x)\ge  0$ for all $t\ge 0$ and $x\ge 0$. 

We are going to  establish existence and desired regularity of the flow $Y$ defined  by \eqref{E33}. Note that as  $J\colon [0,+\infty)\mapsto \mathbb{R}$ has bounded derivative, it is Lipschitz continuous. Let 
$$
\begin{aligned}
A(t)&=\e ^{Z(t)} \prod_{s\le t} \left( 1+\Delta Z(s)\right)\e ^{-\Delta Z(s)}, \qquad t\ge 0, \\
Z(t)&=  \langle g(t),L(t)\rangle_{\mathbb{R}^d}, 
\end{aligned}
$$
Note that $A$ and $Z$ are  positive c\`adl\`ag processes with increasing trajectories.
\begin{proposition}\label{P61}
$(i)$ For each non-negative $x\ge 0$ there is a unique non-negative global in time solution $\left(Y(t, x),\, t\ge 0\right)$ to $\eqref{E33}$. Moreover, for each $t$, $Y(t,x)$ depends  continuously on $x$, $\mathbb{P}$-a.s.. \\
$(ii)$ There exists a version of $Y$ differentiable in $x$ and
\begin{equation}\label{E63}
\frac{\partial Y}{\partial x} (t, x) =\e^{\int_0^t \langle g(s), \nabla J(g(s)Y(s,x))\rangle_{\mathbb{R}^d}\d s} A(t). 
\end{equation}
\end{proposition}
\begin{proof}
By Protter \cite{Protter}, Theorem 37, p. 308, the unique solution to $\eqref{E33}$ depends continuously on $x$.  Since  $\nabla J$  is Lipschitz, by Protter \cite{Protter}, Theorem 39, p. 312, the solution $Y$ is differentiable in $x$ and the derivative satisfies
\begin{align*}
\d \frac{\partial Y}{\partial x}(t, x) &= \langle g(t), \nabla J\left( g(t)Y(t,x)\right)\rangle_{\mathbb{R}^d} \frac{\partial Y}{\partial x}(t,x) \d t \\
&\quad + \frac{\partial Y}{\partial x}(t-, x)\langle g(t),  \d L(t)\rangle _{\mathbb{R}^d} \\
&= \frac{\partial Y}{\partial x}(t-,x) \d X(t),\\
\frac{\partial Y}{\partial x}(0,x)&=1, 
\end{align*}
where
\begin{align*}
X(t) &=\langle g(t), L(t)\rangle_{\mathbb{R}^d} +\int_0^t\langle g(s), \nabla J\left(g(s)Y(s,x)\right)\rangle_{\mathbb{R}^d}\d s\\
&=Z(t)+\int_0^t\langle g(s), \nabla J\left(g(s)Y(s,x)\right)\rangle_{\mathbb{R}^d}\d s. 
\end{align*}
By the Dol\'ean--Dade  formula (see e.g. Protter \cite{Protter}, Theorem 37, p. 84),
$$
\frac{\partial Y}{\partial x}(t,x)=\e ^{X(t)} \prod_{s\le t} \left( 1+\Delta X(s)\right)\e ^{-\Delta X(s)}.
$$
Since $\Delta X(s) =\Delta Z(s)$,  we have \eqref{E63}. 
\end{proof}

We can therefore apply our Theorem \ref{T34}. In consequence, if  $Y(t, x)$, $t\ge 0$, $x\ge 0$, is  the solution to \eqref{E33},  then
$$
u(t, x) =Y\left(t, \int_t^{t+x} r_0(y)\d y \right) ,\quad t\ge 0,\ x\ge 0,
$$
is a solution to \eqref{E31}.

\begin{remark}
In the jump martingale case the problem of solving the corresponding HJM equation is reduced to the solving stochastic differential equation \eqref{E33}. In Proposition \ref{P61} we provide some regularity results on the  solution to \eqref{E33}. Contrary to the gaussian case, we are not able to find an explicite solution to \eqref{E33}. Note that in the ``simplest'' case of $d=1$ and $\nu=\delta_1$, 
$$
J(\xi)= \e ^{-\xi}-1+ \xi. 
$$
This case is treated in Section \ref{SPoisson}. 
\end{remark}

In the next section we will study an interesting example of jump L\'evy process in which we are able to find an explicite formulae  for solutions to auxiliary equation \eqref{E33} and the HJM equation. 

\section{$\alpha$-subordinator case}\label{S6}
Assume that $d=1$ and for simplicity that $g\equiv 1$. Let    $L$ be  the $\alpha$-stable subordinator, where $\alpha \in (0,1)$. Then $L$ is given by \eqref{E26},
and $J(\xi)=- \xi  ^\alpha$ for $\xi\ge 0$, and $J(\xi)=+\infty$ for $\xi<0$.  Assume that for any $x\ge 0$, $r_0(x)\ge 0$. The HJMM equation has the form 
\begin{equation}\label{E71}
\begin{aligned}
\d r(t,x)&= \left[ \frac{\partial r}{\partial x}(t,x)- \frac{\partial }{\partial x} \left(  \int_0^x r(t,x)\right )^{\alpha} \right]\d t \\
&\quad + r(t-,x)\d L(t),\\
r(0,x)&= r_0(x), 
\end{aligned}
\end{equation}
In this case condition \eqref{E61} is not satisfied. However, as in the gaussian case we are able to find en explicite solution to the flow equation \eqref{E33}. Recall, see \eqref{E27}, that $\pi$ is  the Poisson jump measure of $L$. Let 
\begin{equation}\label{E72}
\begin{aligned}
\widetilde L(t)&:= \int_0^t\int_{0}^{+\infty} \left[ (1+\eta)^{1-\alpha} -\eta^{1-\alpha}\right] \pi(\d s, \d \eta),\\
M(t)&:=\e ^{\widetilde L(t)} \prod_{s\le t} \left( 1+\Delta \widetilde L(s)\right)\e ^{-\Delta \widetilde L(s)}, \\
G(t,x)&:=\left( \int_t^{t+x} r_0(s)\d s\right)^{1-\alpha}-(1-\alpha )\int_0^t \frac{\d s}{M(s)}. 
\end{aligned}
\end{equation}
\begin{theorem}\label{T71}
A  global in time strong solution to \eqref{E71} is given by the formula 
\begin{equation}\label{E73}
\begin{aligned}
r(t,x)&= M(t)^{\frac{1}{1-\alpha}}\left[ G(t,x)\right]^{\frac{\alpha}{1-\alpha}}\chi_{[0,+\infty)}(G(t,x))\\
&\quad \times \left(\int_t^{t+x} r_0(s)\d s\right)^{-\alpha} r_0(t+x). 
\end{aligned}
\end{equation} 
\end{theorem}
\begin{proof} Note that  in this case  equation \eqref{E33} for the flow has now the form 
\begin{equation}\label{E74}
\begin{aligned}
\d Y(t,x)&=  -Y(t,x)  ^{\alpha}\d t + Y(t-,x)\d L(t)\\
&= -Y(t,x) ^{\alpha}\d t + Y(t-,x)\int_{0}^{+\infty}  \eta \pi(\d t, \d \eta). 
\end{aligned}
\end{equation}
Clearly $y\equiv 0$ solves the equation
\begin{equation}\label{E75}
\begin{aligned}
\d y(t)&=  -y(t)^{\alpha}\d t + y(t-)\d L(t),\\
y(0)&= 0.
\end{aligned}
\end{equation}
We will not discuss the problem of the uniqueness of a solution to \eqref{E75}. 

Assume that $Y(t,x)>0$. Let $f(y)= y^{1-\alpha}$ and  $V= f(Y)$. Applying It\^o's formula we obtain
\begin{align*}
\d V(t,x)&= \d f(Y(t,x))\\
&= -f'(Y(t,x)) \left( Y(t,x)\right)  ^{\alpha}\d t \\
&\quad + \int_{0}^{+\infty} \left[ f\left(Y(t-,x)+ Y(t-,x)\eta \right)- f(Y(t-,x)\eta ) \right]  \pi(\d t, \d \eta)\\
&= -(1-\alpha) \d t + V(t-, x)\int_0^{+\infty} \left[ (1+\eta)^{1-\alpha}- \eta^{1-\alpha}\right] \pi(\d t,\d \eta)\\
&= -(1-\alpha) \d t + V(t-, x)\d \widetilde L(t). 
\end{align*}
Thus $V$ solves the problem 
\begin{align*}
\d V(t,x)&= -(1-\alpha)\d t + V(t-,x)\d \widetilde L(t), \qquad V(0,x)= x^{1-\alpha}. 
\end{align*}

Note that  $M$ given by \eqref{E72} is  the solution to the homogeneous problem 
$$
\d M(t)=M(t-)\d \widetilde L(t), \qquad M(0)=1. 
$$
We are looking for $V$ of the form 
$$
V(t,x)=c(t,x)M(t).
$$
with absolutely continuous $c$. We have 
$$
\d V(t,x)= M(t)\frac{\partial c}{\partial t}(t,x)\d t+ V(t-,x)\d \widetilde L(t).
$$
Therefore
$$
\frac{\partial c}{\partial t}(t,x)= -(1-\alpha)\frac{1}{M(t)}, \qquad c(0,x)=  x^{1-\alpha}, 
$$
and consequently: 
$$
c(t,x)= x^{1-\alpha}- (1-\alpha )\int_0^t \frac{\d s}{M(s)},
$$
\begin{equation}\label{E76}
V(t,x)= \left[ x^{1-\alpha}-(1-\alpha )\int_0^t \frac{\d s}{M(s)}\right]M(t).
\end{equation}
Clearly $V(t,x)>0$ for small $t$ but it may heat $0$ at some finite time. Therefore 
\begin{align*}
Y(t,x)&=V(t,x)^{\frac{1}{1-\alpha}}\chi_{[0,+\infty)}(c(t,x))\\
&= (c(t,x))^{\frac{1}{1-\alpha}} \chi_{[0,+\infty)}(c(t,x)) M(t)^{\frac{1}{1-\alpha}}. 
\end{align*}
is a (possibly not unique) solution to the flow equation. By Theorem \ref{T34}, $r(t,x)$ is given by \eqref{E73}. 
\end{proof}
\begin{remark} Assume that $V$ is given by \eqref{E76}. Then 
$$
Y(t,x)= \vert V(t,x)\vert ^{\frac{1}{1-\alpha}}, \qquad t\ge 0,\ x\ge 0,
$$
also is a solution to the flow equation 
$$
\begin{aligned}
\d Y(t,x)&=  -\left\vert Y(t,x)\right\vert   ^{\alpha}\d t + Y(t-,x)\d L(t)\\
Y(0,x)&= x, \qquad x\ge 0. 
\end{aligned}
$$
\end{remark}

\begin{remark} Note that in our case the solution is global. This follows also from a general result of Barski and Zabczyk \cite{Barski-Zabczyk}, Chapter 14. 
\end{remark}
\section{The case of Poisson process}\label{SPoisson}
Let us consider the simplest case of $d=1$, $L$ being a Poisson process with intensity $1$. Then, the L\'evy measure $\nu(\d x)= \delta_{1}(\d x)$, and  $J(\xi)= \e^{-\xi}-1$.
In the case of compensated noise $J(\xi)= \e^{-\xi}-1+\xi$.  For simplicity we assume that the volatility coefficient $g\equiv 1$. Since 
$$
\frac{\partial }{\partial x} \left(\e^{-r(t,x)}-1\right)= \frac{\partial }{\partial x} \e^{-r(t,x)}, 
$$
the corresponding  HJM equation has the form 
\begin{equation}\label{E81}
\begin{aligned}
\d r(t,x)&= \left[ \frac{\partial r}{\partial x}(t,x)+ \frac{\partial }{\partial x}  \e^{- r(t,x)}  \right]\d t + r(t-,x)\d L(t),\\
r(0,x)&= r_0(x). 
\end{aligned}
\end{equation}
Equation for the flow has the form 
\begin{equation}\label{E82}
\begin{aligned}
\d Y(t,x)&= \e^{- Y(t,x)}\d t + Y(t-,x)\d L(t),\\
Y(0,x)&= x. 
\end{aligned}
\end{equation}

Note that in the case of compensated noise $\widehat L(t)=L(t)-t$, the HJM equation is 
\begin{equation}\label{E83}
\begin{aligned}
\d r(t,x)&= \left[ \frac{\partial r}{\partial x}(t,x)+ \frac{\partial }{\partial x} \left(  \e^{- r(t,x)}  + r(t,x)\right) \right]\d t \\
&\qquad + r(t-,x)\d \widehat  L(t),\\
r(0,x)&= r_0(x). 
\end{aligned}
\end{equation}
However, \eqref{E82} is also  the corresponding equation for the flow. Therefore we have the following consequence of our general Theorem \ref{T34}. 
\begin{corollary}
Equations \eqref{E81} and \eqref{E83} have the same solutions. 
\end{corollary}

Let $\tau_n$ be the time of $n$-th jump of $L$. Then 
\begin{align*}
Y(t,x)&= \log \left( \e^{Y(\tau_n,x)}+ t-\tau_n\right), \qquad t\in [\tau_n,\tau_{n+1}), \\
Y(\tau_n,x)&= 2 Y(\tau_n-,x)= 2\log\left( \e^{Y(\tau_{n-1},x)}+ \tau_n-\tau_{n-1}\right),\\
Y(0,x)&= x. 
\end{align*}
Consequently, $Z(t,x):= \e^{Y(t,x)}$ satisfies
\begin{align*}
Z(t,x)&= Z(\tau_n,x)+ t-\tau_n,\qquad  t\in [\tau_n,\tau_{n+1}),\\
Z(\tau_n,x)&= \left( Z(\tau_{n-1})+\tau_n-\tau_{n-1}\right)^2, \qquad n\ge 1, \\
Z(0,x)&=Z(\tau_0,x)= \e^{x}. 
\end{align*} 
Theorem \ref{T34} yelds the following result. 
\begin{theorem}
The unique solution to \eqref{E81} or equivalently \eqref{E83} is given by 
$$
r(t,x)= \frac{\d }{\d x} \log \left[ Z\left(t, \int_0^{t+x}r_0(y)\d y\right)\right]. 
$$
\end{theorem}

\end{document}